\documentclass[10pt]{amsart}

\usepackage{amssymb, esint}

\usepackage{url}

\usepackage{tikz}
\newtheorem{theorem}{Theorem}[section]
\newtheorem{lemma}[theorem]{Lemma}
\newtheorem{corollary}[theorem]{Corollary}

\theoremstyle{definition}
\newtheorem{definition}[theorem]{Definition}

\newtheorem{notation}[theorem]{Notation}

\theoremstyle{remark}
\newtheorem{remark}[theorem]{Remark}

\numberwithin{equation}{section}

\newcommand{\R}{\mathbb{R}}

\newcommand{\Z}{\mathbb{Z}}

\newcommand{\D}{\mathcal{D}}
\renewcommand{\P}{\mathcal{P}}
\newcommand{\J}{\mathcal{J}}

\newcommand{\qak}{Q_\alpha^k}

\newcommand{\Holder}{H\"{o}lder }
\newcommand{\Holders}{H\"{o}lder's }

\newcommand{\Hytonen}{Hyt\"{o}nen }
\newcommand{\Hytonenc}{Hyt\"{o}nen, }

\newcommand{\mean}[2]{\langle #2 \rangle_{#1}}

\begin{document}

\title[A Dyadic Gehring Inequality]{A Dyadic Gehring Inequality in Spaces of Homogeneous Type and Applications}


\author{Theresa C. Anderson}
\address{The University of Wisconsin-Madison \\
Madison, WI \\
53703}
\email{tcanderson@math.wisd.edu}

\author{David E. Weirich}
\address{1 University of New Mexico \\
Albuquerque, NM \\
87131}
\email{dweirich@unm.edu}



\date{\today} 



\begin{abstract}
We state a version of Gehring's self improvement Theorem for reverse \Holder weights which is valid for dyadic cubes over spaces of homogeneous type and explore some of the consequences and applications. 
\end{abstract}

\maketitle


\section{Introduction}
Gehring's Theorem is a classical result in harmonic analysis due to F. W. Gehring in \cite{geh1} which gives a remarkable partial reversal of the decreasing nature of the reverse Holder weight classes.  Precisely, for $1 < p < \infty$, we say that a weight (nonnegative locally integrable function) $w$ belongs to the Reverse \Holder $p$ class if there exists a constant $C$ so that for all intervals $I = [a, b]$,
\begin{equation}
\left(\frac{1}{b-a}\int_a^b w(x)^p \, dx\right)^{1/p} \leq C \frac{1}{b-a}\int_a^b w(x) \, dx. \label{into:RHolder}
\end{equation}
It is a trivial consequence of \Holders inequality that if $w$ satisfies \eqref{into:RHolder} for some $p$, then it likewise satisfies \eqref{into:RHolder} for any $1 < q < p$.  Surprisingly though, one can show that there exists $\epsilon > 0$ so that $w$ satisfies \eqref{into:RHolder} for $p + \epsilon$ as well.  This is the well known Gehring Theorem, first proved in \cite{geh1}, and we say it is a self improvement result because we have slightly improved the exponent.
This theorem has many applications such as to the theory of quasi-conformal mappings.

Recent work has gone into proving an analogue to Gehring's Theorem in the more abstract setting of spaces of homogeneous type - quasi-metric spaces equipped with a doubling measure.  In \cite{maa1}, Maasalo showed that the theorem is true in metric spaces with doubling measures provided the measure satisfies a radial decay property.  Then in \cite{and1}, Anderson, \Hytonenc and Tapiola showed that the theorem is true for \emph{weak} Reverse \Holder classes in general spaces of homogeneous type.  What characterizes these classes as weak is that the domain of integration is enlarged on the right hand side of the inequality.  One would hope that the ``strong'' result would soon follow, however in the same paper the authors constructed an explicit counterexample: a weight over a specific space which satisfies a inequality analogous to \eqref{into:RHolder} for $p \leq p_0$ but not for $p > p_0$.

In \cite{per1}, Katz and Pereyra used a decaying stopping time argument to prove Gehring's Theorem for weights over the real line.  In the current paper we adapt this method to show that, in spite of the aforementioned counterexample, a \emph{dyadic} version of the strong Gehring Theorem does indeed hold.
\begin{theorem}[Dyadic Gehring's Theorem in Spaces of Homogeneous Type]
Let $1 < p < \infty$ and $w$ a weight over a space of homogeneous type.  If $w \in RH_p^d$ then $w \in RH_{p+\epsilon}^d$ where $RH_p^d$ denotes the class of weights which satisfy a dyadic reverse \Holder $p$ inequality.
\end{theorem}

We close this article by expanding on the counterexample and presenting a simple proof of a sufficient condition for Gehring's Theorem to hold on spaces of homogeneous type.
This may have been known, but to our knowledge this is the first time that this sufficient condition has appeared in the literature.
For a different proof under slightly different conditions, see \cite{korte1}.
This leads to a few more corollaries.

In Section \ref{section_def} we give the necessary definitions and background.  In Section \ref{section_main} we state the main result of this paper, and give the idea of the stopping time.  In Section \ref{proofSection} we give the proof and in Section \ref{Applications} we explore some differences between the reverse H\"older classes in $\R^n$ and Spaces of Homogeneous Type and expand on the counterexample given in \cite{and1}, leading to some new results.

\subsection{Acknowledgments}
Theresa was supported by NSF DMS-1502464.  David would like to thank Professors Eric Sawyer, David Cruz-Uribe, and Leonid Slavin for their helpful comments and suggestions during the 2015 AMS Spring Sectional Meeting at Michigan State University. 
The authors would like to thank the anonymous reviewer for their insightful comments.
Finally, David would also like to thank his advisor M. Cristina Pereyra for her constant support and encouragement.

\section{Definitions} \label{section_def}

In this section we introduce the basic definitions used in this paper.  Readers already familiar with these definitions may desire to skip to the next section.

\subsection{Spaces of Homogeneous Type}

Here we introduce the so-called spaces of homogeneous type, first defined by Coiffman and Weiss in \cite{coif1}.

\begin{definition}[Quasi-metric space]
Let $X$ be a set, and let $\rho: X \times X \to \R^+\cup \{0\}$ be a function which satisfies all the axioms of a metric except the triangle inequality.  Instead, there exists a constant $\kappa_0 > 0$ such that for all $x, y, z \in X$,
\begin{equation}
\rho(x, y) \leq \kappa_0(\rho(x, z) + \rho(z, y)). \label{quasi-triangle-ineq}
\end{equation}
A function $\rho$ satisfying \ref{quasi-triangle-ineq} is called a \emph{quasi-metric} and $(X, \rho)$ is called a \emph{quasi-metric space}.
\end{definition}

As usual, we denote by $B(x, r) := \{y \in X : \rho(x, y) < r\}$ the open ball centred at $x \in X$ of radius $r > 0$ with respect to $\rho$.

\begin{definition}[Geometrically Doubling] \label{defn-geom-dbl}
Let $(X, \rho)$ be a quasi-metric space.
If there exists a constant $M \geq 1$ such that for any ball $B$ of radius $r$, it is possible to cover $B$ by no more than $M$ balls of radius $r/2$, we say that $(X, \rho)$ is \emph{geometrically doubling.}
\end{definition}

\begin{definition}[Space of Homogeneous Type]
Let $(X, \rho)$ be a quasi-metric space and let $\mu$ be a measure on $X$ which satisfies that 
\begin{itemize}
\item the $\sigma$-algebra of $\mu$-measurable sets contains both the Borel $\sigma$-algebra as well as all open $\rho$-balls,
\item there exists a constant $\kappa_1 > 0$ such that for all balls $B(x, r) \subset X$,
\begin{equation}
\mu(B(x, 2r)) \leq \kappa_1 \cdot \mu(B(x, r)). \label{doubling-measure}
\end{equation}
\item $0 < \mu(B(x, r)) < \infty$ for every $x \in X$ and every $r > 0$.
\end{itemize}
A measure satisfying \eqref{doubling-measure} is said to be a \emph{doubling measure} on $X$ and the tuple $(X, \rho, \mu)$ is called a \emph{space of homogeneous type}.
\end{definition}

\begin{remark}
Remember, ``geometric doubling'' is a property of the metric, while ``doubling'' is a property of the measure.  These two similar terms do not mean the same thing.
\end{remark}

\begin{lemma}[Spaces of Homogeneous Type are Geometrically Doubling]
Let $(X, \rho, \mu)$ be a space of homogeneous type with $\mu$ a nontrivial measure, i.e. $\mu \not\equiv 0$ and $\mu \not\equiv \infty$.  Then $(X, \rho)$ is a geometrically doubling metric space. Moreover the geometric doubling constant $M$ from Definition \ref{defn-geom-dbl} depends only on $\kappa_0$ and $\kappa_1$.
\end{lemma}

This lemma is due to Coifman and Weiss (\cite{coif1}, pg. 68).

\begin{remark}
The converse of the above lemma is \emph{not true}.  In other words, one can equip a geometrically doubling quasi-metric space with a measure which is non-doubling. For example: $\R$ with the usual metric and the Gaussian probability measure.
\end{remark}

For more on the basic properties of Spaces of Homogeneous type, see \cite{kai1}, \cite{mac1}, \cite{coif1}.

\subsection{Existence of Dyadic Cubes}

Of interest in this paper is the analogue to the traditional dyadic cubes we are familiar with in $\R^n$ that were first described by Christ in \cite{christ1} (see also \cite{saw1}).  Here we recall the modern construction due to \Hytonen and Kairema, found in \cite{kai1}.  Notice that this construction is \emph{independent of measure}, i.e. it depends only on the properties of the quasi-metric.  We paraphrase the main result of this paper below, omitting details which are not necessary for the result of the present paper.

\begin{theorem}[Dyadic Cubes] \label{dyadicCubesExist}
Let $(X, \rho)$ be a quasi-metric space which is geometrically doubling.  Then there exists a system (or ``lattice'') of dyadic cubes $\D = \{\qak : k \in \Z, \alpha \in \mathcal{A}_k\}$ where $\mathcal{A}_k$ is an indexing set no larger than countably infinite.  These cubes satisfy the following properties:
\begin{enumerate}
\item Cubes are organized into generations.  For each $k \in \Z$ we can define the $k^\text{th}$ generation $\D_k := \{\qak : \alpha \in \mathcal{A}_k\}$. Furthermore, each generation forms a partition of $X$, i.e.,
$$X = \bigcup_{Q \in \D_k} Q.$$ \label{formsPartition}
\item Cubes are mutually nested. If $k \geq \ell$ then for any $Q \in \D_k$ and $Q' \in \D_\ell$, either $Q \subseteq Q'$ or $Q \cap Q' = \emptyset$. In the case where $Q \subseteq Q'$ we say that $Q$ is a \emph{descendant} of $Q'$. \label{mutualNestedness}
\item Cubes are comparable to balls.  There exist constants $0 < r_0 \leq R_0 < \infty$ and $0 < \delta < 1$ independent of $Q$ so that for every $Q \in \D_k$ there is a point $z \in Q$ where
$$B(z, r_0\delta^k) \subseteq Q \subseteq B(z, R_0\delta^k).$$ \label{cubeBallComparison}
\end{enumerate}
\end{theorem}
\begin{remark}
The dyadic lattice $\D$ may not be unique, and in general will not be (with the exception of contrived examples, such as $X = \{x_0\}$, a single point).  Theorem \ref{dyadicCubesExist} simply gives one such system of cubes.  The proof is constructive, but it is sometimes useful in specific examples to bypass this construction when a more convenient one is readily available. For example, if $X = \R$ with the usual metric then the standard collection of dyadic intervals are a \emph{dyadic structure}, even though the proof may have constructed a different collection.
\end{remark}

It is a simple consequence of properties \ref{formsPartition} - \ref{cubeBallComparison} of Theorem \ref{dyadicCubesExist} that cubes, like balls, will satisfy a doubling property with respect to a doubling measure.

\begin{corollary}[Parent Cubes] \label{parentCorollary}
Let $\D$ be a dyadic lattice for $(X, \rho)$ a geometrically doubling quasi-metric space. For every $Q \in \D_k$, there exists a unique cube $\widehat{Q} \in \D_{k-1}$ so that $Q \subseteq \widehat{Q}$.  We refer to $\widehat{Q}$ as $Q$'s \emph{parent}.  Furthermore, if $(X, \rho, \mu)$ is a space of homogeneous type, there exists a constant $D$ independant of $Q$ so that
\begin{equation}
\mu(\widehat{Q}) \leq D\cdot \mu(Q).
\end{equation}
for all $Q \in \D$.
\end{corollary}

The proof of Corollary \ref{parentCorollary} follows from a straightforward application of the properties of dyadic cubes and of the doubling measure and can be found in the Appendix.

\begin{remark}
We will use the notation $\D(Q) := \{Q' \in \D : Q' \subseteq Q\}$ to refer to the set of all dyadic cubes which are descendants of $Q$. 
\end{remark}

\subsection{Weights}
We use weights (nonnegative locally integrable functions) that belong to both the $A_p$ classes and reverse Holder classes.

\begin{notation}
For an integrable function $f: X \to \R$ and a $\mu$-measurabe set $S \subseteq X$ with $\mu(S) < \infty$ we denote by $\mean{S}{f}$ the mean of $f$ over $S$, i.e.,
\begin{equation*}
\mean{S}{f} := \frac{1}{\mu(S)}\int_S f(x) \, d\mu(x).
\end{equation*}
\end{notation}

\begin{definition}[Reverse \Holder Class] \label{reverseHolderDefn}
Let $(X, \mu)$ be a measure space.  Let $1 < p < \infty$, let $w$ be a weight, and let $\mathcal{S}$ be a family of subsets of $X$.  Suppose there exists a constant $C$ such that for all $S \in \mathcal{S}$
\begin{equation}
\mean{S}{w^p}^{1/p} \leq C \cdot \mean{S}{w}.
\end{equation}
Then we say that $w$ belongs to the \emph{reverse \Holder $p$} class with respect to $\mathcal{S}$, written $w \in RH_q(\mathcal{S})$ and we denote the smallest such $C$ as $[w]_{RH_p(\mathcal{S})}$, called the \emph{reverse \Holder $p$ characteristic of $w$.}  In particular, if $\rho$ is a quasi-metric on $X$ and $\mathcal{S}$ is the collection of all open balls, we say $w$ belongs to the \emph{continuous reverse \Holder $p$} class and write $w \in RH_p$.  Moreover, if  $(X, \rho, \mu)$ has a dyadic structure $\D$ and $\mathcal{S} = \D$ then we say $w$ belongs to the \emph{dyadic reverse \Holder $p$} class and write $w \in RH_p^d$.
\end{definition}

Notice that Definition \ref{reverseHolderDefn} is meaningful whether $\mu$ is a doubling measure or not. 

\begin{definition}
We say that $w$ belongs to the class $A_p$ if 

\[[w]_{A_p} := \sup_B\fint_B wd\mu\left( \fint_Bw^{1-p'}d\mu\right) ^{p-1} <\infty.\]
\end{definition}

There are many different definitions of $A_\infty$, some of which are not equivalent in SHT.  We cite the following, used quite often in recent work due to Fujii and Wilson \cite{Fujii}, \cite{Wilson}.
\begin{definition}
We say a weight $w$ is in the class $A_\infty$ if
\begin{equation}
  [w]_{A_\infty} = \sup_B \frac{1}{w(B)} \int_B M(1_B w) \, d\mu < \infty,
\end{equation}
\end{definition}
Here $B$ is the family of balls.

In the $A_p$ definition, one can switch between balls and dyadic cubes easily by using the sandwich property (3) of the dyadic system of the SHT.  However, with the $A_\infty$ and reverse H\"older conditions, this cannot be done!  The fact that a dyadic Gehring inequality (using dyadic cubes) is true, but the continuous Gehring (using balls) is not crucially displays the problem from carelessly switching between balls and dyadic cubes. 

We have that the reverse H\"older classes decrease in SHT, i.e. $RH_s\subset RH_r$ for $r<s$.  This can be seen using H\"older's inequality.

Also, by following the proof in $\R^n$ from \cite{Grafakos}, we have that in SHT if $w\in A_\infty$ then $w$ is doubling.

However, the fact in $\R^n$ that $w\in RH_p$ implies that $w$ is doubling is no longer true and will be crucially alluded to below.
\section{Main Result} \label{section_main}

In this section we give our main result and begin to build up the framework to support the proof.
This proof could potentially be reworked in the terminology of sparse cubes.
We chose an approach similar to \cite{per1} using the notation of stopping times.
Readers familiar with this terminology can skip to Section \ref{section-Lemmas}.

\subsection{Gehring's Theorem}

The main theorem of this paper is that Gehring's Theorem holds in the dyadic setting for spaces of homogeneous type.

\begin{theorem}[Main Result] \label{mainResult}
Let $(X, \rho, \mu)$ be a space of homogeneous type with dyadic lattice $\D$ where the Lebesgue Differnetiation Theorem holds with respect to cubes in $\D$.
Let $1 < p < \infty$ and let $w \in RH_p^d$.
Then there exists $\epsilon$ depending only on $p$, $w$, $\kappa_0$ and  $\kappa_1$ such that $w \in RH_{p + \epsilon}^d$.
\end{theorem}

\subsection{Decaying Stopping Time}

The proof of \ref{mainResult}, which can be found in Section \ref{proofSection}, relies on a decaying stopping time argument.  We introduce the idea here.  Throughout this section $(X, \rho, \mu)$ is assumed to be a space of homogeneous type, with dyadic structure $\D$.

\bigskip

Let $\mathcal{P}$ denote some property about cubes as sets.
This property may depend on any number of parameters including other cubes. For a fixed cube $Q \in \D$, we denote by $\J(Q) \subsetneq \D(Q)$ a collection of subcubes which are maximal with respect to $\P$.  By maximality, we mean that if $Q' \subseteq Q$ has $\P$, then no descendant of $Q'$ will be included in $\J(Q)$, regardless of whether it has $\P$ or not.  Formally,
\begin{equation}
\J(Q) := \left\{Q' \in \D(Q) : Q' \text{ has } \P \text{ and } Q'' \text{ does not have } \P \,\,\forall Q'' \supsetneq Q'\right\}.
\end{equation}

Primarily, for the purposes of stopping times, we are interested in properties which relate one cube to another.


\begin{definition}[Admissible Property]
Suppose that $\mathcal{P}$ is a property about cubes with respect to another cube.  Then we say $\P$ is \emph{admissible} if for all $Q \in \D$, $Q$ \emph{does not} have $\P$ with respect to itself (as a set).
\end{definition}


For an admissible property set $\mathcal{J}_0(Q) := \{Q\}$.  We now define the collections $\mathcal{J}_n(Q)$ inductively. Let $n > 0$. Define
\begin{equation*}
\mathcal{J}_n(Q) := \bigcup_{Q' \in \mathcal{J}_{n-1}(Q)} \mathcal{J}(Q').
\end{equation*}
The family of collections $\{\mathcal{J}_n(Q)\}_{n \geq 0}$ is called the \emph{stopping time $\mathcal{J}$ for $Q$.}

\begin{definition}[Decaying Stopping Time]
Let $(X, \rho, \mu)$ be a quasi-metric space equipped with a measure which has dyadic structure $\D$ and let $\mathcal{J}$ be a stopping time.  We say that $\mathcal{J}$ is a \emph{decaying stopping time} if and only if there exists $0 < c < 1$ such that for every $Q \in \D$,
\begin{equation}
\sum_{Q' \in \mathcal{J}_1(Q)} \mu(Q') \leq c \mu(Q). \label{decaying-ineq}
\end{equation}
\end{definition}
\begin{remark}
Iterating \ref{decaying-ineq} gives that
\begin{equation}
\sum_{Q' \in \mathcal{J}_n(Q)} \mu(Q') \leq c^n \mu(Q)
\end{equation}
provided $\mathcal{J}$ is decaying.
\end{remark}


\subsection{The Stopping time $\J^w$}

Let us now describe a particular stopping time.  Suppose that $w \in RH_p^d$ for some $1 < p < \infty$.  If $Q$ is a cube, we say that another cube $Q' \subset \D(Q)$ has property $\P^w$ with respect to $Q$ if either $\mean{Q'}{w} \geq \lambda \mean{Q}{w}$ or $\mean{Q'}{w} \leq \lambda^{-1}\mean{Q}{w}$ where $\lambda > 1$ is a fixed parameter.  While this property depends on a weight $w$, a parameter $\lambda$ and a cube $Q$, we only write $\mathcal{P}^w$ (as opposed to, say, $\mathcal{P}^{w, \lambda}_Q$, in order to avoid over-cluttered notation.

Clearly the following lemma is true.

\begin{lemma}
Property $\P^w$ is admissible.
\end{lemma}
\begin{proof}
For any cube $Q$, since $\lambda > 1$, $\mean{Q}{w} < \lambda\mean{Q}{w}$ and $\mean{Q}{w} > \lambda^{-1}\mean{Q}{w}$.  Thus no cube will ever have property $\P^w$ with respect to itself, which implies admissability.
\end{proof}

We define the stopping time $\J^w$ for $Q$ as the stopping time generated by $\P^w$ with respect to $Q$.

\subsection{Lemmas} \label{section-Lemmas}

To prove Theorem \ref{mainResult} we show the following two lemmas:

\begin{lemma} \label{JwIsDecaying}
If the stopping $\J^w$ described above is decaying then Theorem \ref{mainResult} holds.
\end{lemma}

\begin{lemma} \label{DecayingImpliesGehring}
The stopping time $\J^w$ is decaying provided the parameter $\lambda$ is chosen large enough.
\end{lemma}

It is thus sufficient to prove Lemmas \ref{JwIsDecaying} and \ref{DecayingImpliesGehring}.  The following fact will be useful for both proofs.
\begin{lemma} \label{usefulFact}
Let $Q' \in \J^w(Q)$. Then $\mean{Q'}{w} \leq D\lambda\mean{Q}{w}$ where $D$ is the constant from Corollary \ref{parentCorollary}.
\end{lemma}
\begin{proof}
By the maximality condition for stopping times, since $Q' \in \J^w(Q)$, its parent $\widehat{Q'} \not\in \J^w(Q)$.  This means that $\lambda^{-1}\mean{Q}{w} < \mean{\widehat{Q'}}{w} < \lambda\mean{Q}{w}$.  Thus,
\begin{align*}
\mean{Q'}{w} &= \frac{1}{\mu(Q')} \int_{Q'} w\,d\mu \leq \frac{1}{\mu(Q')} \int_{\widehat{Q'}} w\,d\mu \\
&\leq \frac{D}{\mu(\widehat{Q'})} \int_{\widehat{Q'}} w\,d\mu = D\mean{\widehat{Q'}}{w} < D \lambda \mean{Q}{w}.
\end{align*}
 
\end{proof}

\begin{corollary} \label{usefulCorollary}
Suppose $Q' \in \J^w_n(Q)$.  Then $\mean{Q'}{w} \leq (D\lambda)^n\mean{Q}{w}$.
\end{corollary}
\begin{proof}
Let $Q^0 := Q' \in \J^w_n(Q)$.  By definition, there exists $Q^1 \in \J^w_{n-1}$ so that $Q^0 \in \J^w(Q^1)$.  Continuing on in this fashion, for all $1 \leq i \leq n$ there exists $Q^i \in \J^w_{n-i}$ so that $Q^{i - 1} \in \J^w(Q^i)$.  With this notation, $Q^n = Q$. Iterating the result of Lemma \ref{usefulFact} $n$ times gives that
\begin{align*}
\mean{Q'}{w} &= \mean{Q^0}{w} \leq D\lambda \mean{Q^1}{w} \leq (D\lambda)^2 \mean{Q^2}{w} \\
&\leq \cdots \leq (D\lambda)^n\mean{Q^n}{w} = (D\lambda)^n\mean{Q}{w}.
\end{align*}

\end{proof}

The following will also be useful.

\begin{lemma} \label{aeGoodBound}
For almost every $x \in X$ (with respect to the measure $\mu$), $\lambda^{-1}\mean{Q}{w} \leq w(x) \leq \lambda \mean{Q}{w}$ for $x \not\in \cup_{Q' \in \J^w(Q)} Q'$.
\end{lemma}
\begin{proof}
Let $x \in Q$ such that $x \not\in Q'$ for all $Q' \in \J^w(Q)$.  Let $k_0$ be $Q$'s generation, i.e. $Q \in \D^{k_0}$ and define $Q_x^k$ as the cube belonging to generation $\D^k$ with $x \in Q_x^k$ for $k \geq k_0$.  So $Q_x^k \not\in \J^w(Q)$ for all $k \geq k_0$, thus by definition of property $\P^w$,
\begin{equation*}
\lambda^{-1}\mean{Q}{w} \leq \mean{Q_x^k}{w} \leq \lambda \mean{Q}{w}.
\end{equation*}
By the Lebesgue Differentiation Theorem, the limit as $k \to \infty$ of the center expression goes to $w(x)$ a.e. with respect to the measure $\mu$.
\end{proof}
In the previous proof we used the Lebesgue Differentiation Theorem.
A dyadic version of this theorem is asserted in \cite{Martinglaes}.
However, this issue is a bit delecate.
We refer to \cite{mit1} for a discussion of these matters.
To avoid these issues we simply include the theorem as a hypothesis of Theorem \ref{mainResult}.

\section{Proofs} \label{proofSection} 

In this section we present the proofs of Lemmas \ref{JwIsDecaying} and \ref{DecayingImpliesGehring}, thus establishing Theorem \ref{mainResult}.

\begin{proof}[Proof of Lemma \ref{JwIsDecaying}]
Fix $\lambda$ large, precisely how large to be determined later.  For now it suffices to enforce that $\lambda > 3$.  For a cube $Q \in \D$ let $\J^w$ be the stopping time for $Q$.  Since the property $\P^w$ with respect to $Q$ has two mutually exclusive stopping conditions, we can split $\J^w(Q)$ into two disjoint parts:
\begin{equation*}
\J^w(Q) = \{Q' \in \D(Q) : \mean{Q'}{w} \geq \lambda \mean{Q}{w}\} \sqcup \{Q' \in \D(Q) : \mean{Q'}{w} \leq \lambda^{-1} \mean{Q}{w}\}
\end{equation*}
where by $\sqcup$ we mean the disjoint union, i.e., the union of two disjoint sets.
We let $\{Q^\lambda_i\}_i$ be an enumeration of the subcubes in the first part and $\{Q^{1/\lambda}_i\}_i$ be an enumeration of the subcubes in the second part.  We then write $Q$ as the disjoint union of the three subsets
\begin{equation}
Q = B^\lambda \sqcup B^{1/\lambda} \sqcup G
\end{equation}
with ``bad parts'' $B^\lambda := \cup_i Q^\lambda_i$ and $B^{1/\lambda} := \cup_i Q^{1/\lambda}_i$ (so called since the mean is either too large or too small on these parts) and ``good part'' $G := Q \setminus (B^\lambda \cup B^{1/\lambda})$.  It follows from Lemma \ref{aeGoodBound} that
\begin{equation*}
\lambda^{-1}\mean{Q}{w} \leq w(x) \leq \lambda\mean{Q}{w} \hspace*{4mm} \text{a.e. } x \in G.
\end{equation*}

Suppose that the desired lemma is false, that is, suppose that $\J^w$ is \emph{not} decaying.  This would imply that for each $0 < c < 1$ we can find a cube $Q \in \D$ such that
\begin{equation*}
\sum_{Q' \in \J^w(Q)} \mu(Q') = \mu(Q \setminus G) > c \cdot \mu(Q)
\end{equation*}
implying that
\begin{equation*}
(1 - c) > \frac{\mu(G)}{\mu(Q)}.
\end{equation*}
In other words, the ratio of the measure of the good part to the measure of the whole cube can be made arbitrarily small.

Choose $Q \in \D$ such that $\mu(G) \leq \mu(Q)/(3\lambda)$.  Then
\begin{align}
\int_G w\, d\mu &\leq \int_G \lambda \mean{Q}{w}\,d\mu = \mu(G)\cdot\lambda\mean{Q}{w} \nonumber \\
&=\mu(G)\cdot\frac{\lambda}{\mu(Q)}\int_Q w\,d\mu \leq \frac{1}{3}\int_Q w\,d\mu \label{GisLessThanOneThirdQ}
\end{align}
and
\begin{align}
\int_{B^{1/\lambda}} w\,d\mu &\leq \mu(B^{1/\lambda})\cdot \lambda^{-1} \mean{Q}{w} \leq \lambda^{-1}\frac{\mu(B^{1/\lambda})}{\mu(Q)}\int_Q w\,d\mu \nonumber \\
&\leq \lambda^{-1}\int_Q w\,d\mu < \frac{1}{3}\int_Q w\,d\mu. \label{BsmallIsLessThanOneThirdQ}
\end{align}
Inequalities \eqref{GisLessThanOneThirdQ} and \eqref{BsmallIsLessThanOneThirdQ} together imply that
\begin{align}
\int_{B^\lambda} w\,d\mu &= \int_{Q \setminus (G \cup B^{1/\lambda})} w \, d\mu = \int_Q w\, d\mu - \int_G w\, d\mu - \int_{B^{1/\lambda}} w\, d\mu \nonumber \\
&> \int_Q w\,d\mu - \frac{1}{3} \int_Q w\,d\mu - \frac{1}{3} \int_Q w\,d\mu = \frac{1}{3}\int_Q w\,d\mu. \label{BbigIsGreaterTHanOneThirdQ}
\end{align}
We can also see that
\begin{align}
\mean{B^\lambda}{w} &= \frac{1}{\mu(B^\lambda)}\sum_i\int_{Q^\lambda_i} w\, d\mu = \frac{1}{\mu(B^\lambda)}\sum_i \mu(Q^\lambda_i) \mean{Q^\lambda_i}{w} \nonumber \\
&\leq \frac{1}{\mu(B^\lambda)}\sum_i \mu(Q^\lambda_i) D\lambda \mean{Q}{w} = D\lambda\mean{Q}{w} \label{usedUsefulFact1}
\end{align}
where in \eqref{usedUsefulFact1} we used Lemma \ref{usefulFact}.  We use \eqref{usedUsefulFact1} and \eqref{BbigIsGreaterTHanOneThirdQ} to get a lower bound on the measure of $B^\lambda$:
\begin{align}
\mu(B^\lambda) &= \frac{1}{\mean{B^\lambda}{w}}\int_{B^\lambda} w\, d\mu \geq \frac{1}{3\mean{B^\lambda}{w}}\int_{Q} w\, d\mu \nonumber \\
&\geq \frac{1}{3D\mean{Q}{w}}\int_{Q} w\, d\mu = \frac{1}{3D\lambda}\mu(Q) \label{MeasureOfBBigLowerBound}
\end{align}

We will now use this lower bound to establish a contradiction. Observe that
\begin{align}
\int_Q w^p \, d\mu &\geq \int_{B^\lambda} w^p \, d\mu
= \sum_i \int_{Q^\lambda_i} w^p \, d \mu \nonumber \\
&\geq \sum_i \frac{1}{\mu(Q^\lambda_i)^{p-1}}\left(\int_{Q^\lambda_i} w \, d\mu \right)^p \label{weUsedHolder}\\
&= \sum_i \mu(Q^\lambda_i) \mean{Q^\lambda_i}{w}^p
\geq \lambda^p \sum_i \mu(Q^\lambda_i)\mean{Q}{w}^p \label{byDefnofBBig}\\
&= \lambda^p \mu(B^\lambda)\mean{Q}{w}^p 
\geq \frac{1}{3D}\lambda^{p-1}\mu(Q)\mean{Q}{w}^p \label{byTheBound}
\end{align}
where in \eqref{weUsedHolder} follows from the \Holder inequality, \eqref{byDefnofBBig} by the definition of $B^\lambda$, and \eqref{byTheBound} from \eqref{MeasureOfBBigLowerBound}.  Dividing both sides by $\mu(Q)$ and taking the $1/p$ power gives that
\begin{equation}
\mean{Q}{w^p}^{1/p} \geq \left(\frac{1}{3D}\lambda^{p-1}\right)^{1/p}\mean{Q}{w}.
\end{equation}
We thus contradict that $w \in RH^d_p$, provided that $\lambda$ is chosen large enough so that $\lambda > (3D[w]^p_{RH^d_p})^{1/(p-1)}$.
\end{proof}

\begin{remark}
The preceding proof was a proof by contradiction.  While we demonstrated that the decaying constant $c$ does exists, we have no guarantee on the size of this constant.
\end{remark}

\begin{proof}[Proof of Lemma \ref{DecayingImpliesGehring}]

Let $Q \in \D$ be any cube.  We define the $n^\text{th}$ ``good'' and ``bad'' sets as
\begin{align*}
B_n(Q) &:= \bigcup_{Q' \in \mathcal{J}_n^w(Q)} Q'\,\,;\,\,\,\, n \geq 0, \\
G_n(Q) &:= B_{n-1}(Q) \setminus B_n(Q)\,\,;\,\,\,\, n > 0.
\end{align*}
Notice that $B_0(Q) = Q = \sqcup_n G_n(Q)$. By the Lemma \ref{JwIsDecaying}, we can choose $\lambda > 1$ sufficiently large to ensure that $\mathcal{J}^w$ is decaying. So there exists $0 < c < 1$ so that
\begin{equation*}
\mu(B_n(Q)) \leq c^n\mu(Q)\,\,\,\,;\,\,\,\,\forall Q \in \D.
\end{equation*}
Our first goal will be to establish that
\begin{equation}
\int_{G_n(Q)}w^p\, d\mu \leq a^{n-1}\int_Q w^p \, d\mu \label{aConstantGoal}
\end{equation}
for a constant $0 < a < 1$ depending only on $p$, $c$, $[w]_{RH_p^d}$, $\kappa_0$ and $\kappa_1$.
First, we consider some properties of $G_1(Q)$.  We know by Lemma \ref{aeGoodBound} that
\begin{equation*}
\lambda^{-1}\mean{Q}{w} \leq w(x) \hspace*{4mm} \text{a.e. } x \in G_1(Q),
\end{equation*}
and that
\begin{equation*}
\mu(G_1(Q)) \geq (1 - c) \mu(Q).
\end{equation*}
Using these two facts, we conclude that
\begin{align}
\int_{G_1(Q)} w^p \, d\mu & \geq \int_{G_1(Q)} \frac{1}{\lambda^p}\mean{Q}{w}^p \, d\mu 
= \frac{\mu(G_1(Q))}{\lambda^p}\mean{Q}{w}^p \nonumber \\
&\geq \frac{(1- c)\mu(Q)}{\lambda^p} \mean{Q}{w}^p
\geq \frac{(1- c)\mu(Q)}{\lambda^p[w]_{RH^d_p}} \mean{Q}{w^p} \nonumber \\
&= \frac{(1 - c)}{\lambda^p[w]_{RH^d_p}} \int_Q w^p \, d\mu \label{hasAconstantLessThanOne}
\end{align}
Notice that the domain of integration for the far right hand side of inequality \eqref{hasAconstantLessThanOne} is a subset of the domain of integration of the far left hand side.  In fact, $\mu(G_1(Q)) < \mu(Q)$.  Set
\begin{equation*}
(1 - a) := \frac{(1 - c)}{\lambda^p[w]_{RH^d_p}} \in (0, 1).
\end{equation*}
We observe that this constant $a$ depends only on $p$, $c$, $[w]_{RH_p^d}$, $\kappa_0$ and $\kappa_1$.
In particular, we observe that $a$ is independent of $Q$.
We now iterate this result.
We observe (in order to abuse) that
\begin{equation*}
G_n(Q) = \bigsqcup_{Q' \in \mathcal{J}^w_{n-1}(Q)} G_1(Q').
\end{equation*}
This allows us to easily see that
\begin{align*}
\int_{G_n(Q)} w^p \, d\mu & = \sum_{Q' \in \mathcal{J}^w_{n-1}(Q)} \int_{G_1(Q')} w^p \, d\mu \\
&\geq \sum_{Q' \in \mathcal{J}^w_{n-1}(Q)} (1 - a) \int_{Q'} w^p\, d\mu \\
&= (1 - a) \int_{B_{n-1}(Q)} w^p \, d\mu.
\end{align*}
With this, we now have that
\begin{align}
\int_{B_n(Q)} w^p \, d\mu &= \int_{B_{n-1}(Q)} w^p \, d\mu - \int_{G_{n}(Q)} w^p \, d\mu \nonumber \\
&\leq \int_{B_{n-1}(Q)} w^p \, d\mu - (1 - a)\int_{B_{n-1}(Q)} w^p \, d\mu \nonumber \\
&= a\int_{B_{n-1}(Q)} w^p \, d\mu. \label{iterateThis}
\end{align}
Since $G_n(Q) \subseteq B_{n-1}(Q)$, iterating \eqref{iterateThis} $n - 1$ times gives \eqref{aConstantGoal}.

Fix $\epsilon > 0$ (determined later). Using what was shown above,
\begin{align}
\int_Q w^{p+\epsilon} \, d \mu &= \sum_{n = 1}^\infty \int_{G_n(Q)} w^{p + \epsilon} \, d \mu \nonumber \\
&\leq \mean{Q}{w}^{\epsilon} \sum_{n = 1}^\infty (D\lambda)^{n\epsilon} \int_{G_n(Q)} w^p \, d \mu \label{usedUsefulCorollary} \\
&\leq \mean{Q}{w}^{\epsilon} \sum_{n = 1}^\infty (D\lambda)^{n\epsilon} a^{n-1}\int_Q w^p \, d \mu
\end{align}
where in line \eqref{usedUsefulCorollary} we used Corollary \ref{usefulCorollary}. From here, we choose $\epsilon$ small enough so that $(D\lambda)^\epsilon < a^{-1}$, which is possible since $0 < a < 1$.  Then the sum
\begin{equation*}
\sum_{n = 1}^\infty (D\lambda)^{n\epsilon} a^{n-1} =: A < \infty.
\end{equation*}
Therefore, dividing both sides by $\mu(Q)$ gives that
\begin{align*}
\mean{Q}{w^{p+\epsilon}} &\leq A\mean{Q}{w}^\epsilon \mean{Q}{w^{p}} \\
&\leq A[w]_{RH^d_p}^p \mean{Q}{w}^{p+\epsilon}.
\end{align*}
Since the constant $A$ depended only on $p$, $w$, $\kappa_0$ and $\kappa_1$ we can conclude that $w \in RH_{p + \epsilon}^d$.
\end{proof}

\begin{remark}
By examining the constants in the proof, we can actually see that $\epsilon < \frac{1}{[w]_{RH^d_{p + \epsilon}}-1}$.
\end{remark}

It is worth noting that the only time the doubling condition on the measure $\mu$ was used was in Lemma \ref{usefulFact}.  With this in mind we can state the following corollary.

\begin{corollary} \label{corollary1}
Let $(X, \rho, \mu)$ be a quasi-metric measure space with $\mu$ a measure which may or may not be doubling and some dyadic structure $\D$.  Let $1 < p < \infty$ and let $w \in RH_p^d$ be a weight such that there exists constants $C_1>D$ so that for all cubes $Q \in \D$:
\begin{equation}
\mean{Q}{w} \leq C_1\mean{\widehat{Q}}{w} \label{doublingWeight}
\end{equation}
Then there exists $\epsilon > 0$ such that $w \in RH_{p+\epsilon}^d$. (Recall $\widehat{Q}$ denotes the unique parent cube of $Q$.)
\end{corollary}



\begin{remark}
It is easy to confuse a doubling weight with a doubling measure.  However, these are not the same thing.  there exist weights which are not doubling over measures which are, and non-doubling measures can support doubling weights.  In light of this, it is important to take care when using this terminology.
\end{remark}

\section{Consequences and Applications}\label{Applications}
We have shown that in any space of homogeneous type a dyadic strong Gehring does hold, but from the counterexample in \cite{and1}, a strong continuous Gehring using the metric balls does not hold. It turns out that the key property that this counterexample lacks is doubling of the measure $w$.  
Recall that the weight $w$ is doubling if 
\[
w(2B) \leq Cw(B)
\]
for all balls,
and that $w$ is dyadic doubling if 
\[
w(\hat{Q}) \leq Cw(Q)
\]
for all cubes $Q\in \mathcal{D}$.  We use the notation $Db$ to indicate the class of doubling weights.

In \cite{KLPW1}, the authors prove that
\begin{equation}
RH_p\cap Db= \bigcap_{j = 1}^{J_0} \left(RH_p(\D^{(j)}) \cap Db(\D^{(j)}) \right).
\end{equation}
where they use $J_0$ distinct dyadic systems in an SHT.  In other words, for doubling weights, the continuous reverse \Holder class is equal to the intersection of finitely many dyadic reverse \Holder classes that are also dyadic doubling.  In $\R^n$ note that $RH_p$ implies doubling (continuous), but dyadic $RH_p$ does not necessarily imply dyadic doubling.  This is no longer true in an SHT.  Even though we have shown that dyadic Gehring does hold in any SHT, this does not imply that continuous Gehring does.    

The counterexample to strong continuous Gehring in \cite{and1} is in fact not doubling.  Since the counterexample is $RH_p$ for certain values of p, we must no longer have that $RH_p$ implies doubling, as is true in $\R^n$.  This is an important distinction between $\R^n$ and SHT.

We will now show directly that the counterexample is not doubling.  We briefly recall the details below but refer the reader to \cite{and1} as well.

\begin{theorem}
	The counterexample in \cite{and1} is not doubling.  Explicitly, we show that there exists a sequence of balls $B_j$ such that $w(2B_j) \geq 4$ for all $j$ but that $w(B_j) \to 0$ as $j\to\infty$.
\end{theorem}
\begin{proof}
	We first recall some details from the counterexample.
	Define a metric space $(X,d)$ as follows.  Take $\R^2$ with the $l^{\infty}$ metric so the balls are actually squares.  Let our space $X$ be the "haircomb space" defined as $X = A\cup\bigcup_{j\in\mathbb{N}}W_j$ with 
	\begin{equation*}
	A = \{(u,0):u\in\R\},\quad
	U = \{(u,\frac12u):u\in(0,1]\},\quad
	V = \{(1,v):v\in[\frac12,1]\},\quad
	\end{equation*}
	and $W_j:=U\cup V+(10j,0) = U_j \cup V_j$.

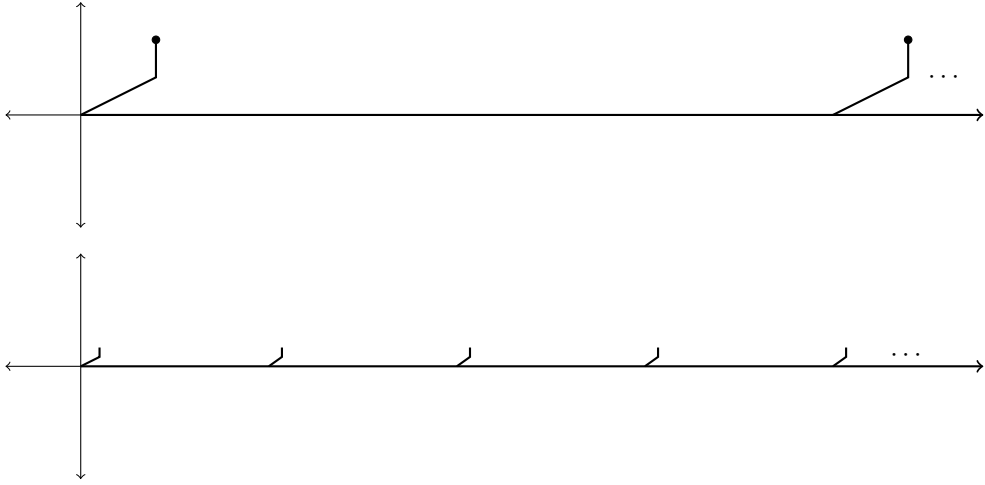
\begin{figure}[h]
\centering

\begin{tikzpicture}
\draw[thick, ->] (0, 0) -- (12, 0);
\draw[<-] (-1, 0) -- (0, 0);
\draw[<->] (0, -1.5) -- (0, 1.5);
\draw[thick] (0, 0) -- (1, 1/2) -- (1, 1);
\draw[thick] (10, 0) -- (11, 1/2) -- (11, 1);
\draw[fill=black] (1, 1) circle (.05);
\draw[fill=black] (11, 1) circle (.05);
\node at (11.5, .5) {$\cdots$};
\end{tikzpicture}

\bigskip

\begin{tikzpicture}
\draw[thick, ->] (0, 0) -- (12, 0);
\draw[<-] (-1, 0) -- (0, 0);
\draw[<->] (0, -1.5) -- (0, 1.5);
\draw[thick] (0, 0) -- (1/4, 1/8) -- (1/4, 1/4);
\draw[thick] (2.5, 0) -- (2.675, 1/8) -- (2.675, 1/4);
\draw[thick] (5, 0) -- (5.175, 1/8) -- (5.175, 1/4);
\draw[thick] (7.5, 0) -- (7.675, 1/8) -- (7.675, 1/4);
\draw[thick] (10, 0) -- (10.175, 1/8) -- (10.175, 1/4);
\node at (11, .15) {$\cdots$};
\end{tikzpicture}
\caption{The haircomb counterexample.  Above: Zoomed in. Below: Zoomed out to show repetition.}
\end{figure}

	We use the $l^{\infty}$ metric and the arc-length measure.
	The weight $f_h$ is defined as
	\begin{eqnarray*}
		f(x) = \left\{ \begin{array}{cl}
			1, & \text{if } x \in A \\
			\varepsilon_j, & \text{if } x \in V_j \\
			\min\{1,\varepsilon_j g(u)\}, & \text{if } x = (10j + u,\frac{1}{2}u) \in U_j
		\end{array} \right. .
	\end{eqnarray*}
	where $\varepsilon_j \to 0^+$, $\varepsilon_j\leq 1$, $h(t):=t^{-\alpha}\log^{-1}(e/t)$ for some $0<\alpha<1$ and  $g(t) = \max\{h(t),1\}$.  Note $f_h \leq 1$ everywhere.
	
	Recall that the authors of \cite{and1} showed that this weight was in $RH_p$ if and only if $p\leq 1/\alpha$, which implies the failure of the strong Gehring inequality.
	
	Now we construct the sequence of balls $B_j$.  The idea is to have $B_j$ pick up mass only on one of the comb teeth, but to have $2B_j$ pick up a sizable mass of the line $A$ which is more heavily weighted.  Since the measure of the comb teeth depends on $\epsilon_j$ which heads to 0, the measure of each subsequent $B_j$ will decrease.  Let $B_j$ be the ball centered at $(10j+1, 1/2)$ with radius $1/2$.  
	Now \[f_h(2B_j) = \int_{2B_j}f_h(x)d\mu = \int_{A\cap B_j}f_h(x)dx +\int_{U_j}f_h(x)du+\int_{V_j}f_h(x)dv
	\]
	\[
	\geq 4+\epsilon_j\cdot 1/2 \geq 4.
	\]
	Finally, we show that $f_h(B_j) \to 0$ as $j\to\infty$.
	\[f_h(B_j) = \int_{B_j}f_h(x)d\mu = \int_{U_j\cap B_j}f_h(x)du+\int_{V_j}f_h(x)dv
	\]
	\[
	\leq \int_{U_j\cap B_j}\sup{h(u),1}du+\epsilon_j\cdot 1/2 \leq C_{\alpha}\epsilon_j
	\]
	since $h(u)$ is integrable ($h\in L^1[0,1]$), so the integral over $U_j$ is bounded by a constant $C_{\alpha}$.
	Since $\epsilon_j$ is chosen such that $1\geq \epsilon_j \geq 0, \epsilon_j \to 0$, we have that $f_h(B_j) \to 0$.
	
	Therefore, $f_h$ is not a doubling weight.
\end{proof}

The failure of doubling in the counterexample led to this simple proof of this apparently new fact that doubling of $w$ is indeed sufficient for Gehring in SHT.

\begin{theorem}
Gehring's inequality holds in Spaces of Homogeneous Type if $w$ is a doubling weight.
\end{theorem}
\begin{proof}
Let $w\in RH_p$.  Then we have that $w\in RH_p^\sigma$, the weak reverse Holder class, that is \[
  \left( \fint_B w^q \right)^{1/q} \leq [w]_{RH_q}^\sigma \fint_{\sigma B} w 
  \]
 for some $\sigma >\kappa_0$ \cite{and1}.  Therefore $w\in RH_{p+\epsilon}^\sigma$ by the weak Gehring inequality in \cite{and1}, so we have
\[
\left( \frac{1}{\mu(B)}\int_B w^{p+\epsilon}\right) ^{1/p+\epsilon} \leq C\frac{1}{\mu(B)}\int_{\sigma B}w \leq C D_w \frac{1}{\mu(B)}\int_Bw
\]
where we have used in the last step that $w(\sigma B) \leq D_w w(B)$ due to the doubling of $w$, and the constant $D_w$ depends on $\sigma$ and the doubling constant of $w$.  Thus, $w\in RH_{p+\epsilon}$ as was to be shown.
\end{proof}

This theorem provides some counterexamples to well-known and frequently used relationships between the reverse Holder and the $A_p$ weight classes.

The following were originally in \cite{CN}.
\begin{corollary}
In $\R^n$ we have that $w\in A_p$ if and only if $w\in RH_s$ for some $s$.  This is not true in SHT as there exists a $w\in RH_s$ such $w$ is not doubling, so therefore $w\notin A_\infty$, so $w\notin A_p$ for any $p$.
\end{corollary}

\begin{corollary}
In $\R^n$ we have that $w\in A_\infty$ if and only if $w\in RH_1$.  Again, referencing the above corollaries, this is not true in SHT.
\end{corollary}

\section{Appendix}

For interested readers we give the proof of Corollary \ref{parentCorollary}.

\begin{lemma}[Doubling for General Radii] \label{doublingGeneralR}
Let $(X, \rho, \mu)$ be a space of homogeneous type. If $x \in X$ and $R > r > 0$ then
\begin{equation}
\mu(B(x, R)) \leq \kappa_1^{\log_2 \lceil R/r \rceil} \cdot \mu(B(x, r)),
\end{equation}
\end{lemma}
\begin{proof}
By the doubling property,
\begin{align}
\mu((B(x,R)) &\leq \kappa_1 \cdot \mu(B(x,R/2)) \nonumber \\
&\leq \kappa_1^2 \cdot \mu(B(x,R/4)) \nonumber \\
&\leq \cdots \nonumber \\
&\leq \kappa_1^n \cdot \mu(B(x,R\cdot 2^{-n}))
\end{align}
Choose $n$ so that $r/2\leq R2^{-n} < r$.  
\end{proof}

\begin{lemma}[Distant Balls Lemma] \label{distantBalls}
Let $x, y \in X$ and set $R := \rho(x, y)$.  Then for all $r > 0$,
\begin{equation}
\mu(B(y, r)) \leq \kappa_1^{\log_2\left(\frac{\kappa_0(R + r)}{r}\right)} \cdot \mu(B(x, r)).
\end{equation}
\end{lemma}
\begin{proof}
Let $x, y \in X$ and $r > 0$. Set $R = \rho(x, y)$.  We wish to cover the ball $B(y, r)$ with a ball centered at $x$.  To do this, the radius $\kappa_0(R + r)$ suffices.  To see this, suppose that $z \in B(y, r)$.  Then
\begin{align*}
\rho(x, z) &\leq \kappa_0(\rho(x, y) + \rho(y, z)) \\
&=\kappa_0(R + r)
\end{align*}
which implies that $z \in B(x, \kappa_0(R + r))$.  Thus,
\begin{align*}
B(y, r) &\subseteq B(x, \kappa_0(R + r)) \\
\mu(B(y, r)) &\leq \mu(B(x, \kappa_0(R + r)) \\
&\leq \kappa_1^{\log_2\left(\frac{\kappa_0(R + r)}{r}\right)} \cdot \mu(B(x, r))
\end{align*}
where the last line follows from Lemma \ref{doublingGeneralR}.
\end{proof}

\begin{proof}[Proof of Corollary \ref{parentCorollary}]
Let $Q \in \D_k$ be a cube, with parent cube $\widehat{Q} \in \D_{k-1}$  Then there exists balls $B_1 := B(z_1, r_0\delta^k) \subseteq Q$ and $B_2 = B(z_2, R_0\delta^{k-1}) \supseteq \widehat{Q}$.  Therefore, 
\begin{align}
\mu(\widehat{Q}) &\leq \mu(B_2) \nonumber \\
&\leq \kappa_1^{\log_2 \lceil R_0/(r_0\delta) \rceil}\cdot \mu(B(z_2, r_0\delta^{k})) \label{byGenDouble} \\
&\leq \kappa_1^{\log_2\left(\frac{\kappa_0(R_0\delta^{k-1} + r_0\delta^k)}{r_0\delta^k}\right)}\cdot \kappa_1^{\log_2 \lceil R_0/(r_0\delta) \rceil} \cdot \mu(B_1) \label{byDistantBall} \\
& \leq \kappa_1^{\log_2\left(\frac{\kappa_0(R_0\delta^{k-1} + r_0\delta^k)}{r_0\delta^k}\right)}\cdot \kappa_1^{\log_2 \lceil R_0/(r_0\delta) \rceil} \cdot \mu(Q) \nonumber\\
\end{align}

where \eqref{byDistantBall} follows from the Distant Balls Lemma, and \eqref{byGenDouble} follows from doubling for general radii. 
\end{proof}
\bibliographystyle{amsplain}
\bibliography{bibliography}
\end{document}